\documentclass{amsart}

\usepackage[all]{xy}
\usepackage{amsmath}
\usepackage{hyperref}
\usepackage{amsfonts,graphics,amsthm,amsfonts,amscd,latexsym}
\usepackage{epsfig}
\usepackage{flafter}
\usepackage{mathtools}
\usepackage{comment}
\usepackage{stmaryrd}
\usepackage{tabularx}
\usepackage{pst-node}
\usepackage{auto-pst-pdf}
\usepackage{tikz-cd}
\hypersetup{
    colorlinks=true,    
    linkcolor=blue,          
    citecolor=blue,      
    filecolor=blue,      
    urlcolor=blue           
}

\usepackage{pgfplots}
\pgfplotsset{compat=1.15}
\usepackage{mathrsfs}
\usepackage{tikz}
\usetikzlibrary{graphs,positioning,arrows,shapes.misc,decorations.pathmorphing}

\tikzset{
    >=stealth,
    every picture/.style={thick},
    graphs/every graph/.style={empty nodes},
}

\tikzstyle{vertex}=[
    draw,
    circle,
    fill=black,
    inner sep=1pt,
    minimum width=5pt,
]
\usepackage[position=top]{subfig}
\usepackage{amssymb}
\usepackage{color}

\calclayout

\usetikzlibrary{decorations.pathmorphing}
\tikzstyle{printersafe}=[decoration={snake,amplitude=0pt}]

\newcommand{\supp}{\operatorname{supp}}

\newcommand{\pp}{\mathbb{P}}

\newcommand{\qq}{\mathbb{Q}}

\newcommand{\rr}{\mathbb{R}}
\newcommand{\cc}{\mathbb{C}}

\usepackage{tikz}
\usetikzlibrary{matrix,arrows,decorations.pathmorphing}
\usetikzlibrary{arrows}

\definecolor{uuuuuu}{rgb}{0.26666666666666666,0.26666666666666666,0.26666666666666666}

  \newtheorem{theorem}{Theorem}[section]
  \newtheorem{lemma}[theorem]{Lemma}

  \newtheorem{definition}[theorem]{Definition}
  \newtheorem{example}[theorem]{Example}

\newtheorem{remark}[theorem]{Remark}

\theoremstyle{remark}

\numberwithin{equation}{section}



\begin{document}

\title[Bott-Chern complexity of K\"ahler pairs]{Bott-Chern complexity of K\"ahler pairs}

\author[C.~Hacon]{Christopher Hacon}
\address{Department of Mathematics, University of Utah, Salt Lake City, UT 84112,
USA}
\email{hacon@math.utah.edu}

\author[J.~Moraga]{Joaqu\'in Moraga}
\address{UCLA Mathematics Department, Box 951555, Los Angeles, CA 90095-1555, USA
}
\email{jmoraga@math.ucla.edu}

\author[J.I.~Y\'a\~nez]{Jos\'e Ignacio Y\'a\~nez}
\address{UCLA Mathematics Department, Box 951555, Los Angeles, CA 90095-1555, USA
}
\email{yanez@math.ucla.edu}
\thanks{The first author was partially supported by  NSF research grant   DMS-2301374 and
by a grant from the Simons Foundation SFI-MPS-MOV-00006719-07. 
The second author was partially supported by NSF research grant DMS-2443425.}

\begin{abstract}
We introduce the Bott-Chern complexity 
of a compact K\"ahler pair $(X,B)$.
This invariant compares $\dim(X)$, $\dim H^{1,1}_{\rm BC}(X)$ and the sum of the coefficients of $B$.
When $(X,B)$ is Calabi--Yau, we show that
its Bott-Chern complexity is non-negative.
We prove that the Bott-Chern complexity of a Calabi--Yau compact K\"ahler pair
$(X,B)$
is at least three whenever $X$ is not projective.
Furthermore, we show this value is optimal and is achieved by certain singular non-projective K3 surfaces.
\end{abstract}
\maketitle

\setcounter{tocdepth}{1}
\tableofcontents

\section{Introduction}

In the recent years, there has been substantial progress in our understanding of the 
birational geometry of K\"ahler varieties. 
Many results of the projective Minimal Model Program have been generalized to the K\"ahler setting. 
For example, there has been exciting progress in the threefold K\"ahler MMP \cite{HP15,HP16,DH20,DHY23,DH24}, the threefold Abundance Conjecture \cite{CHP16,DO24,DO23}, the Canonical Bundle Formula \cite{HP24}, and the relative setting for projective morphisms \cite{Fuj22,LM22,DHP24}.

From the perspective of projective algebraic geometry,
the complexity is an invariant that allows us to measure how far a 
variety is from being toric. 
The complexity of a $\qq$-factorial projective log pair $(X,B)$ measures the difference between $\dim(X)+\rho(X)$
and the sum of the coefficients of $B$.
More precisely, we set 
\[
c(X,B):=\dim X + \rho(X)-|B|, 
\]
where $|B|$ stands for the sum of the coefficients of $B$.
Whenever $(X,B)$ is a projective Calabi--Yau pair, i.e., satisfies $K_X+B\sim_\qq0 $ and has log canonical singularities, the complexity
is a non-negative number~\cite{BMSZ18}.
Furthermore, if $c(X,B)<1$ then $X$ is a projective toric variety.
In~\cite{ELY25}, the authors show that $X$ is of cluster type, whenever 
$c(X,B)=1$. Cluster type varieties are a special kind of projective varieties that are compactifications of algebraic tori (see, e.g.,~\cite[Definition 2.26]{EFM24}). 
In~\cite{MY25}, the authors study Calabi--Yau pairs of complexity two
and develop a method to determine whether they are of cluster type. 
In the projective setting, the complexity has also been connected with the topology of dual complexes of Calabi--Yau pairs~\cite{MM24}
and with the existence of birational conic fibrations~\cite{Mor24a}. 

An important question for K\"ahler varieties, is to have criteria to decide whether a K\"ahler variety $X$ is projective. Some examples of such results are Kodaira's criterion, where a K\"ahler manifold $X$ is projective if $H^0(X,\Omega^2_X) = 0$, or Moishezon's criterion, in which a normal K\"ahler variety $X$ is projective if $X$ is Moishezon, meaning that it admits a big line bundle and has rational singularities \cite{Nam02}.
In this article, we make a connection between the complexity for K\"ahler pairs and the projectivity of varieties. 
We introduce the {\em Bott-Chern complexity};
given a compact complex K\"ahler pair $(X,B)$, 
its Bott-Chern complexity is the dimension of $X$ plus
the first Bott-Chern cohomology of $X$ minus
the sum of the coefficients of $B$ (see Definition~\ref{def:BC-comp}).
Our first theorem shows that 
the Bott-Chern complexity of non-projective pairs $(X,B)$ is at least three whenever $-(K_X+B)$ is a nef divisor. 

\begin{theorem}\label{introhm:BCC}
Let $X$ be a compact K\"ahler variety
and $(X,B)$ be a log canonical pair.
Assume that $X$ is strongly $\qq$-factorial and
that $-(K_X+B)$ is nef. 
Then, the following statements hold. 
\begin{enumerate}
\item If
$
\dim X + h^{1,1}_{\rm BC}(X) -|B| < 3,
$
then $X$ is a projective variety.
\item If 
$\dim X + h^{1,1}_{\rm BC}(X)-|B|=3 $
and $X$ is non-projective, then the base of the MRC fibration of $X$ 
is a singular non-projective K3 surface $W$
of Picard rank zero and $h^{1,1}_{\rm BC}(W)=1$.
\end{enumerate}
\end{theorem}

\begin{remark}
\em{
In case (2) of Theorem \ref{introhm:BCC} we will find a bimeromorphic model $Y$
of $X$ such that the MRC fibration $Y \rightarrow W$ is a morphism.
On this bimeromorphic model, the fibers of the MRC fibration
are projective toric varieties (see Theorem~\ref{BCC-decomposition}).
Furthermore, the variety $W$ is a singular K3 surface
$W$ of Picard rank zero and $h^{1,1}_{\rm BC}(W)=1$.
}
\end{remark}

In Example~\ref{ex:1} and Example~\ref{ex:2}, we show that Theorem~\ref{introhm:BCC}.(2) already happens among non-projective K\"ahler surfaces with $B=0$.
These K\"ahler surfaces are constructed as singular models
of non-projective degenerations of extremal elliptic K3 surfaces. 
We use the classification due to Shimada and Zhang~\cite{SZ01} 
to construct these examples. 

In order to prove Theorem~\ref{introhm:BCC}, we will need to study a similarly defined invariant; the {\em fine complexity} in which $H^{1,1}_{\rm BC}(X)$ is replaced with the span of the components of $B$ in $N_{n-1}(X)_\rr$ (see Definition~\ref{def:decomp}). 
We show that whenever $X$ is a compact K\"ahler variety
and $-(K_X+B)$ is a nef divisor, the fine complexity $\overline{c}(X,B)$
gives an upper bound for the dimension of the base of the MRC fibration of $X$. More precisely, we prove the following: 

\begin{theorem}\label{introthm:base-of-MRC}
Let $X$ be a strongly $\qq$-factorial compact K\"ahler variety. 
Let $(X,B)$ be a log canonical pair with $-(K_X+B)$ nef. 
Then, the base $Z$ of the MRC fibration of $X$ has dimension at most $\overline{c}(X,B)$. 
\end{theorem}

In Theorem~\ref{introthm:base-of-MRC}, when we mention the MRC fibration of $X$, we mean the MRC fibration defined from any smooth bimeromorphic model of $X$.
In order to prove the previous theorem, we will use the MRC 
fibration in the K\"ahler setting~\cite{CH20}, the 
relative Moishezon property for the MRC fibration~\cite{CH24}, 
and the uniruledness criteria via non-pseudoeffectivity
of the canonical divisor proved recently by W. Ou~\cite{Ou25}. 
In particular, we conclude that whenever a compact K\"ahler pair $(X,B)$ has small fine complexity, then the base of its MRC fibration is likely to be a point.  
Using this statement, we show a projectivity criterion for compact K\"ahler varieties with small fine complexity.

\begin{theorem}\label{introthm:fine-BCC}
Let $X$ be a strongly $\qq$-factorial compact K\"ahler variety. 
Let $(X,B)$ be a log canonical pair 
with $-(K_X+B)$ nef. 
Then, the following statements hold: 
\begin{enumerate}
    \item we have $\overline{c}(X,B)\geq 0$, 
    \item if $\overline{c}(X,B)<1$, then $X$ is a projective toric variety, and  
    \item if $\overline{c}(X,B)<2$, then $X$ is a projective variety with $H^{1,1}_{\rm BC}(X)={\rm Pic}(X)_\rr$.
\end{enumerate}
\end{theorem}

Theorem~\ref{introthm:fine-BCC} implies that all the results for small values of the fine complexity~\cite{BMSZ18,ELY25,MY25} are still valid in the compact K\"ahler setting. 
In~\cite{ELY25}, the authors show that $X$ is Fano type provided that $c(X,B)<2$ and $X$ is projective. Note that this statement is not valid for the fine complexity as shown by considering elliptic curves. However, Theorem~\ref{introthm:fine-BCC}.(3) is still valid for compact complex K\"ahler varieties of fine complexity strictly less than two.

\subsection*{Acknowledgements}
The second and third authors would like to thank Joshua Enwright for discussions on the complexity near one. 
The second author thanks Brendan Hassett and Alexander Kusnetzov for discussions on K3 surfaces. 

\section{Preliminaries}

We work over the field of complex numbers $\mathbb{C}$.
In this article, a \emph{variety} is an irreducible and reduced
complex space.

\subsection{K\"ahler varieties} In this subsection, we recall the definition
of K\"ahler varieties and notions of singularities of pairs. 
For more details, we refer the reader to \cite{Dem12,HP16} and references therein.

An $\rr$-divisor $D$ is a finite sum $D = \sum a_iD_i$, 
with $D_i$ prime Weil divisors, and $a_i \in \rr$. 
If $X$ is a normal variety, we can define the 
\emph{canonical sheaf $\omega_X$} as 
\[
\omega_X := \left(\bigwedge^{\dim X} \Omega_X^1\right)^{**}.
\]
By abuse of notation, we will write $K_X$ and use the additive divisor notation, 
even though $K_X$ might not correspond to a Weil divisor on $X$.
A \emph{sub-pair} $(X,B)$ is the data of a normal analytic variety $X$
and a $\qq$-divisor $B$ on $X$ such that $K_X + B$ is $\qq$-Cartier.
If $B$ is effective, we say that $(X,B)$ is a \emph{pair}.
We define the singularities of $(X,B)$ as in~\cite{KM98}. For similar definitions for K\"ahler generalized pairs, see \cite{DHY23}.

\begin{definition}
{\em
Let $X$ be a compact normal variety. We say that $X$ is \emph{$\qq$-factorial} if for every prime Weil divisor $D$ there exists an integer $m \geq 1$ such that $mD$ is Cartier, and there exists an integer $k\geq 1$ such that $\omega_X^{[k]} := (\omega_X^{\otimes k})^{**}$ is a line bundle. We say that $X$ is \emph{strongly $\qq$-factorial} if for every reflexive sheaf of rank one $\mathcal{L}$ on $X$, there exists an integer $m\geq 1$ such that $\mathcal{L}^{[m]}$ is a line bundle.
}
\end{definition}

Being strongly $\qq$-factorial is preserved by the steps of the Minimal Model Program.

\begin{lemma}[c.f.~{\cite[Lemma 2.5]{DH20}}]\label{lem:strong-Q-fact}
    Let $X$ be a compact variety, and $(X,B)$ be a strongly $\qq$-factorial dlt pair. If $X \dashrightarrow X'$ is a $(K_X + B)$-divisorial contraction or flip, then $X'$ is strongly $\qq$-factorial.
\end{lemma}

\begin{definition}
{\em
    A variety $X$ is \emph{K\"ahler} if there exists a positive closed real $(1,1)$-form $\omega$ such that the following holds:
    for every point $x\in X$ there exists an open subset $U \ni x$ 
    and an embedding $i_U \colon U \rightarrow V$ 
    into an open subset $V$ of $\cc^N$, and a strictly plurisubharmonic $C^\infty$-function $f\colon V \to \rr$
    such that 
    \[
        \omega|_{U\cap X^{\rm sm}} = (i\partial \overline{\partial} \,f)|_{U\cap X^{\rm sm}},
    \] 
    where $X^{\rm sm}$ is the smooth locus of $X$.
}
\end{definition}

The Minimal Model Program preserves the K\"ahler condition.
We use Remark~\ref{lem:mmp-kahler} below implicitly throughout the paper.

\begin{remark}
\label{lem:mmp-kahler}
{\em Let $f:X\to Y$ 
be a projective morphism of compact normal complex varieties. It is well known that if 
$Y$ is K\"ahler then $X$ is K\"ahler. 
In particular if $X\dasharrow X'$ is a flip or divisorial over $Y$ contraction then $X'$ is also K\"ahler. To see this, note that if $X\to Z$ is a flipping or divisorial contraction over $Y$, then $Z$ is projective over $Y$ and hence K\"ahler, and if $X\to Z$ is a flipping contraction over $Y$ and $X^+\to Z$ the corresponding flip, then $X^+\to Z$ is a projective morphism and so $X^+$ is K\"ahler.
}
\end{remark}

Let $X$ be a normal compact variety. Consider the
\emph{Bott-Chern coholomology} $H^{1,1}_{\rm BC}
(X)$ of real $(1,1)$-forms with local potentials (see \cite[Definition 3.1]{HP16}). The Bott-Chern cohomology $H^{1,1}_{\rm BC}(X)$ plays the role of 
$N^1(X)_\rr$ in projective geometry. In particular, 
if $X$ is a K\"ahler variety with rational singularities, then
$H^{1,1}_{\rm BC}(X) \subset H^2(X,\rr)$ (see \cite[Eq (3)]{HP16}), 
so we can define an intersection product for $H^{1,1}_{\rm BC}(X)$ via the cup product of
$H^2(X,\rr)$. For the definition of \emph{nef} and \emph{pseudoeffective} class, see~\cite[Definition 2.2(vi)]{HP16}.

\begin{definition}
{\em
    We say that a compact variety $X$ is 
    \emph{Moishezon} if the trancendence degree of its field of meromorphic function is the dimension of $X$.
}
\end{definition}

\subsection{Maximally rationally connected fibration}
In this subsection, we recall the notion
of maximally rationally connected fibration (MRC fibration). 

Let $X$ be a compact K\"ahler manifold, then the MRC of $X$ is an almost holomorphic fibration $X\dasharrow Z$ such that the general fiber is rationally connected and the dimension of $Z$ is
maximal among all the fibrations of this type, and the base of this fibration $Z$ is not uniruled \cite[Remark 6.10]{CH20}. Note that the MRC is only defined up to bimeromorphic equivalence.
In particular, we may assume that $Z$ is smooth and by \cite{Ou25}, the canonical divisor $K_Z$ is pseudo-effective.
If $X$ is a normal compact K\"ahler variety, then the MRC is defined as the MRC of any resolution of $X$.

Recall that by \cite{HM07} if $X$ has dlt singularities then the fibers of any resolution $X'\to X$ are rationally chain connected and $X$ is rationally connected if and only if it is rationally chain connected. This fails for log canonical singularities as shown by a cone over an elliptic curve. Suppose that $\nu :X'\to X$ is a resolution such that $f:X'\to Z$ is a morphism birational to the MRC where $Z$ is smooth, $z\in Z$ a very general point, $F'=f^{-1}(z)$, and $F=\nu (F')$. Since $z\in Z$ a very general point, there are no rational curves on $Z$ containing $z$. Since all the fibers of $F'\to F$ are rationally chain connected, they must be contracted by $f$. By the rigidity lemma it follows that $X\dasharrow Z$ is also an almost holomorphic fibration. 

It is well known that if $X$ is a smooth compact K\"ahler rationally connected manifold, then $H^2(\mathcal O _X)=H^0(\Omega ^2_X)=0$ and hence $X$ is projective. Similarly, if $(X,B)$ is a compact rationally connected K\"ahler dlt pair and $X'\to X$ is a resolution, then by what we have mentioned above $X'$ is rationally connected and hence projective. Since $X$ has rational singularities, then it follows from \cite{Nam02} that $X$ is also projective.
By \cite{CH24}, it is known that if $(X,B)$ is a compact  K\"ahler klt pair, then there exists a model $X'\to Z$ of the MRC that is a projective morphism. By what we have observed above, $X\dasharrow Z$ is an almost holomorphic fibration whose general fiber $F$ is projective. 

By the previous discussion, we have the following lemma.

\begin{lemma}\label{lem:MRC-klt-type}
Let $X$ be a compact complex K\"ahler variety with klt type singularities.
Then, the MRC fibration $X\dashrightarrow Z$ is an almost holomorphic fibration
whose general fiber $F$ is a projective variety.
\end{lemma}

\subsection{Complexity} In this subsection, we recall the notion of complexity
of pairs and recall some lemmata. For more results regarding complexity
we refer the readers to~\cite[\S 2.4]{BMSZ18}.

\begin{definition}\label{def:decomp}
{\em 
Let $B$ be an effective divisor on a compact K\"ahler variety $X$.
A {\em decomposition} $\Sigma$ of the divisor $B$ is an expression of the form $\sum_{i=1}^{k}b_iB_i\leq B$ 
where each $B_i$ is a Weil effective divisor
and each $b_i$ is a positive real number. 
The {\em fine complexity} of $(X,B;\Sigma)$
with respect to the decomposition $\Sigma$ is defined to be 
\[
\overline{c}(X,B;\Sigma):=\dim X + \dim_\rr \langle \Sigma\rangle - \sum_{i=1}^k b_i 
\]
where $\langle \Sigma \rangle$ is the span of the $B_i$'s in the space of $\rr$-Weil divisors modulo numerical equivalence 
and $|B|:=\sum_{i=1}^k b_i$.
The {\em fine complexity} of $(X,B)$, denoted by $\overline{c}(X,B)$, is the minimum
among all the fine complexities of $(X,B;\Sigma)$ 
with respect to all possible decompositions $\Sigma$. 
} 
\end{definition}

\begin{definition}\label{def:BC-comp}
{\em 
Let $X$ be a compact K\"ahler variety and $B$ be an effective divisor on $X$.
A decomposition $\Sigma=\sum_{i=1}^k b_iB_i\leq B$ is said to be a {\em $\qq$-Cartier
decomposition} if each Weil divisor $B_i$ is $\qq$-Cartier.
The {\em Bott-Chern complexity} of $(X,B;\Sigma)$
with respect to a $\qq$-Cartier decomposition $\Sigma$ is defined to be 
\[
c_{\rm BC}(X,B;\Sigma):=\dim X +  h^{1,1}_{\rm BC}(X)-\sum_{i=1}^k b_i. 
\]
The {\em Bott-Chern complexity} of $(X,B)$ is defined to be the minimum among all
the Bott-Chern complexities $c_{\rm BC}(X,B;\Sigma)$ for all possible $\qq$-Cartier decompositions
$\Sigma$ of $B$. 
}
\end{definition}

The following lemmata are well-known in the algebraic setting, see, e.g.,~\cite[Lemma 3.32]{MS21}.
In the compact K\"ahler setting the proof is verbatim. 

\begin{lemma}\label{lem:Q-fact-dlt}
Let $X$ be a compact K\"ahler variety
and $(X,B)$ be a log pair. 
Let $\pi\colon (Y,B_Y)\rightarrow (X,B)$ be a strongly $\qq$-factorial dlt modification of $(X,B)$.
Then, we have $\overline{c}(Y,B_Y) \leq \overline{c}(X,B).$ 
\end{lemma}

\begin{lemma}\label{lem:restriction-to-fibers}
Let $X$ be a compact K\"ahler variety 
and $(X,B)$ be a log pair.
Let $\phi \colon X \rightarrow Y$ be a fibration.
Assume that all the components of $B$ are horizontal over $Y$. Let $F$ be a general fiber and $B_F$ be the restriction of $B$ to $F$. Then, we have 
$\overline{c}(F,B_F)\leq \overline{c}(X,B)-\dim Z$.
\end{lemma}

\section{The fine complexity}
In this section, we prove that the fine complexity of a compact complex K\"ahler Calabi--Yau pair $(X,B)$
is non-negative. Furthermore, if such value is less than two, then $X$ is a projective Fano type variety.
First, we bound the dimension of the base of the MRC fibration of a compact K\"ahler pair.

\begin{proof}[Proof of Theorem~\ref{introthm:base-of-MRC}]
Let $X$ be a compact complex log canonical K\"ahler variety of dimension $n$, and
let $(X,B)$ be a log canonical pair with $-(K_X+ B)$ nef.
Let $\pi\colon (Y,B_Y)\rightarrow (X,B)$ be a strongly $\qq$-factorial dlt modification (see, e.g.,~\cite[Theorem 1.6]{HP24}). 
By Lemma~\ref{lem:Q-fact-dlt}, we know that 
$\overline{c}(Y,B_Y)\leq \overline{c}(X,B)$. 
Let $\phi\colon Y\dashrightarrow Z$ be the MRC fibration of $Y$ (see, e.g.,~\cite[Remark 6.10]{CH20}). The compact complex K\"ahler variety
$Z$ is not uniruled and so $K_Z$ is pseudo-effective by~\cite{Ou25}. 

First, we argue that all the components of $B_Y$ are horizontal over $Z$. 
Let $p \colon Y'\rightarrow Y$ be a resolution of the indeterminacy of the MRC fibration $\phi$ so that $\phi':= \phi\circ p$ is a holomorphic map.
Write $p^*(K_Y+B_Y)=K_{Y'}+B_{Y'}+E_{Y'}-F_{Y'}$
to be the log pull-back of $(Y,B_Y)$ to $Y'$ 
where $B_{Y'}$ is the strict transform $B_Y$ in $Y'$
and $E_{Y'}$ and $F_{Y'}$ are $p$-exceptional effective divisors without common components.
As the MRC fibration is defined over a dense open subset, we may assume that the general fibers of $\phi$ and $\phi'$ are isomorphic and so both $E_{Y'}$ and $F_{Y'}$ are vertical over $Z'$.
Let $B_v$ be the sum of the vertical components of $B_Y$. 
Assume, by the sake of contradiction, that $B_v\neq 0$. 
Let $B'_v$ be the strict transform of $B_v$ in $Y'$. 
Denote by $B'_h$ the divisor $B_{Y'}+E_{Y'}-B'_v$. 
Then, the divisor $F_{Y'}-B'_v$ is not pseudo-effective. Indeed, if $\omega_{Y}$ is a K\"ahler form on $Y$ then $(F_{Y'}-B'_v)\cdot p^*\omega_Y^{n-1} =
-B_v \cdot \omega_Y^{n-1}<0$.
Let $F'$ be a general fiber of $\phi'$.
 Note that the restriction of
\[
K_{Y'}+B'_h - p^*(K_Y+B_Y) = F_{Y'}-B'_v 
\]
to $F'$ equals $F_{Y'}|_{F'} \geq 0$.
Therefore, by \cite[Theorem 2.2]{HP24}, we know that $K_{Y'/Z'}+B'_h$ is pseudo-effective.
Since $K_{Z'}$ is pseudo-effective, we conclude that 
$K_{Y'}+B'_h$ is pseudo-effective. 
Hence, $F_{Y'}-B'_v$ is pseudo-effective which is impossible. Therefore $B'_v=0$. 

Let $(F,B_F)$ be the pair obtained by restricting 
$(Y,B_Y)$ to the general fiber $F$ of $\phi\colon Y\dashrightarrow Z$. Then, $F$ is a compact K\"ahler variety with klt type singularities. Therefore, $F$ is K\"ahler, rationally connected, with rational singularities. Thus $F$ is projective. 
By Lemma~\ref{lem:restriction-to-fibers}, we conclude that 
\[
\overline{c}(F,B_F) \leq \overline{c}(Y,B_Y)-\dim Z \leq \overline{c}(X,B)-\dim Z. 
\]
By~\cite[Theorem 1.2]{BMSZ18}, we know that 
$\overline{c}(F,B_F)\geq 0$, so we conclude that 
$\dim Z \leq \overline{c}(X,B)$. 
\end{proof}

Now, we turn to prove the projectivity of compact K\"ahler pairs with small fine complexity.

\begin{proof}[Proof of Theorem~\ref{introthm:fine-BCC}]
Let $X$ be a strongly $\qq$-factorial compact complex K\"ahler variety of dimension $n$, 
and let $(X,B)$ be a log canonical pair with $-(K_X+B)$ nef. 
Assume that $\overline{c}(X,B)<2$.
Let $\pi\colon (Y,B_Y)\rightarrow (X,B)$ be a strongly $\qq$-factorial dlt modification (see, e.g.,~\cite[Theorem 1.6]{HP24}). 
By Lemma~\ref{lem:Q-fact-dlt}, we know that 
$\overline{c}(Y,B_Y)\leq \overline{c}(X,B)<2$. 
Let $\phi\colon Y\dashrightarrow Z$ be the MRC fibration of $Y$ (see, e.g.,~\cite[Remark 6.10]{CH20}). 
By Theorem~\ref{introthm:base-of-MRC}, we conclude that either $\dim Z =0$ or $\dim Z=1$. 
In the first case, $Y$ is a projective variety.
We argue that $Y$ is a projective variety in the second case as well.
Indeed, as $\dim Z=1$ and $K_Z$ is pseudo-effective, then $Z$ is a smooth curve of positive genus and in particular $\phi\colon Y \dashrightarrow Z$ is a morphism. 
By~\cite[Theorem 1.2]{CH24}, $\phi$ is bimeromorphic to a projective morphism $\phi'\colon Y'\rightarrow Z$. Therefore, $Y'$ is projective and so $Y$ is Moishezon, with rational singularities, and K\"ahler. Thus, $Y$ is projective as well.
As the fibers of $Y\rightarrow Z$ are rationally connected
and $H^{1,1}_{\rm BC}(Z)={\rm Pic}(Z)_\rr$, we conclude that 
$H^{1,1}_{\rm BC}(Y)={\rm Pic}(Y)_\rr$ holds as well (see \cite[Lemma 2.42]{DH20}). 
Thus, $X$ is a projective variety
and $H^{1,1}_{\rm BC}(X)={\rm Pic}(X)_\rr$ holds.

Finally, we observe that if $\overline{c}(X,B)<2$, then $X$ is a projective variety
and so~\cite[Theorem 1.2]{BMSZ18} applies to prove (1) and (2). 
\end{proof}

\section{The Bott-Chern complexity}

The following Theorem implies Theorem~\ref{introhm:BCC} by taking the decomposition given by the irreducible components of $B$.

\begin{theorem}\label{BCC-decomposition}
Let $X$ be a compact K\"ahler variety
and $(X,B)$ be a log canonical pair.
Assume that $X$ is strongly $\qq$-factorial and that $-(K_X+B)$ is a nef divisor. 
Then, the following statements hold. 
\begin{enumerate} 
\item If $c_{\rm BC}(X,B)<3$, then $X$ is projective. 
\item If $c_{\rm BC}(X,B)=3$ and $X$ is not projective, 
then there is a dlt modification $(X',B')\rightarrow (X,B)$, a small modification $Y'\dashrightarrow X'$ 
and an MRC fibration $Y'\rightarrow W$ for which the general 
fiber is a projective toric variety and the base $W$ is a singular
K3 surface of Picard rank zero and $h^{1,1}_{\rm BC}(W)=1$. 
\end{enumerate} 
\end{theorem}

\begin{proof}
First, assume that $c_{\rm BC}(X,B)<3$.
Let $\Sigma$ be a decomposition that computes 
the Bott-Chern complexity.
We consider the fine complexity of $(X,B;\Sigma)$. 
If $\overline{c}(X,B;\Sigma)<2$, then by Theorem~\ref{introthm:fine-BCC}
we conclude that $X$ is a projective variety. 
If $\overline{c}(X,B;\Sigma)\geq 2$, then we conclude that 
$\langle \Sigma\rangle = H^{1,1}_{\rm BC}(X)$. Thus, we have that 
$H^{1,1}_{\rm BC}(X)={\rm Pic}(X)_\rr$ and hence $X$ is a projective variety. 

Now, we assume that $c_{\rm BC}(X,B)=3$.
We argue that $\overline{c}(X,B)=2$ in this case.
The argument is similar to the previous paragraph. 
Let $\Sigma$ be the decomposition of $B$ that computes
the Bott-Chern complexity of $(X,B)$.
We consider the fine complexity of $(X,B;\Sigma)$. 
If $\overline{c}(X,B;\Sigma)<2$, then by Theorem~\ref{introthm:fine-BCC}
we conclude that $X$ is a projective variety. 
If $\overline{c}(X,B;\Sigma)>2$, then we conclude that
$\langle \Sigma \rangle = H^{1,1}_{\rm BC}(X)$. 
Thus, we have $H^{1,1}_{\rm BC}(X)={\rm Pic}(X)_\rr$ 
and hence $X$ is a projective variety. 
Thus, from now on, we may assume that $\overline{c}(X,B)=2$ holds. 

Now, we turn to argue that $\dim Z \leq 2$.
Let $(X',B')\rightarrow (X,B)$ be a strongly $\qq$-factorial dlt modification.
Note that $-(K_{X'}+B')$ is nef. 
As $X$ and $X'$ are strongly $\qq$-factorial
we know that $\rho(X'/X)$ equals $r$ the number of prime exceptional divisors of $X'$ over $X$. Indeed, the prime exceptional divisors $E_1,\dots, E_r$ are $\qq$-Cartier, 
if they were linearly dependent over $X$, then $\sum_{i=1}^r e_iE_i\equiv_X 0$ and by the negativity lemma we would get $\sum_{i=1}^r e_iE_i$ (see~\cite[Lemma 2.8]{DH}).
Furthermore, we know that $h^{1,1}_{\rm BC}(X')-h^{1,1}_{\rm BC}(X)$ equals $r$
the number of prime exceptional divisors of $X'$ over $X$. 
Let $\Sigma$ be a decomposition of $B$ that computes the Bott-Chern complexity of $(X,B)$. 
Write $\Sigma = \sum_{i=1}^k b_iB_i$ for the $\qq$-Cartier 
decomposition of $B$. 
Then, we have that 
$\Sigma'=\sum_{i=1}^k b_i\psi^{-1}_*B_i + \sum_{i=1}^r E_i \leq B'$
a $\qq$-Cartier decomposition of $B'$. Moreover the equalities
\[
c_{\rm BC}(X',B';\Sigma') = c_{\rm BC}(X,B;\Sigma)=3
\text{ and }
\overline{c}(X',B',\Sigma')=\overline{c}(X,B;\Sigma)=2
\]
hold. As $X'$ has klt type, we have a well-defined 
MRC fibration $\pi\colon X'\dashrightarrow Z$ which is a morphism
over a dense open subset of $Z$ (see Lemma~\ref{lem:MRC-klt-type}).
Proceeding as in the proof of Theorem~\ref{introthm:fine-BCC}, 
we observe that all the components of $B'$ are horizontal over $Z$.
Let $B_F$ be the restriction of $B'$ to $F$. 
Let $\Sigma_F$ be the restriction of $\Sigma'$ to the general fiber of 
$X'\dashrightarrow Z$.
Note that $F$ is a projective variety, 
as it is K\"ahler, has rational singularities, and 
it is rationally connected. 
Then, we have 
\[
\overline{c}(F,B_F;\Sigma_F) \leq \overline{c}(X',B';\Sigma')-\dim Z \leq 2-\dim Z. 
\]
As $F$ is projective and $-(K_F+B_F)$ is nef, 
by~\cite[Theorem 1.2]{BMSZ18} we know that $\overline{c}(F,B_F;\Sigma_F)\geq 0$ 
and so the dimension of $Z$ is at most two. 

If $Z$ is a point, then $X'$ is rationally connected with rational singularities.
Therefore, we have $H^{1,1}_{\rm BC}(X')={\rm Pic}(X')_\rr$ and $X$ is projective, leading to a contradiction.
If $Z$ is a curve, then $X'\dashrightarrow Z$ is indeed a morphism with rationally connected fibers. 
We conclude that $H^{1,1}_{\rm BC}(X')={\rm Pic}(X')_\rr$ (see \cite[Lemma 2.42]{DH20}).
So both $X'$ and $X$ are projective, leading to a contradiction. 

From now on, we assume that $\pi\colon X'\dashrightarrow Z$ is an MRC fibration 
to a compact complex K\"ahler surface and that the equalities 
\[
c_{\rm BC}(X',B';\Sigma')=3 \text{ and } 
\overline{c}(X',B';\Sigma')=2
\]
hold.
Let $\pi'\colon X''\rightarrow Z$ be a resolution of interdeminancy
of the bimeromorphic map $\pi$.
Let $f\colon X''\rightarrow X'$ be the associated projective bimeromorphic morphism. 
Let $E$ be the reduced exceptional divisor of $f$.
By possibly replacing $Z$ with a higher bimeromorphic model, we may assume that 
$(Z,\pi'(E))$ and $(X'',B''+E)$ are log smooth. 
The morphism $\pi'$ is a projective morphism
between compact complex K\"ahler manifolds. 
Write $f^*(K_{X'}+B')=K_{X''}+B''$.
Note that $M'':=-(K_{X''}+B'')$ is nef. 
However, $(X'',B'')$ may not be a pair.
Indeed, $B''$ may have some negative coefficients. 
Nevertheless, all the negative coefficients of $B''$
happen along prime components of $B''$ which are vertical over $Z$. 
The previous statement follows from the fact that
$X'\dashrightarrow Z$ is a morphism over an open subset of $Z$.
Therefore, we may apply the canonical bundle formula
for the generalized sub-pair $(X'',B''+M'')$ over $Z$ (see~\cite[Theorem 0.3]{HP24}).
We obtain a generalized surface sub-pair $(Z,B_{Z}+M_{Z})$
which is sub-log Calabi--Yau, i.e., it has generalized log canonical
singularities and $K_{Z}+B_{Z}+M_{Z}\equiv 0$. 
Again, the divisor $B_{Z}$ may have negative coefficients. 
We may assume that over every prime component $P$ of $B_{Z}^{=1}$
there is a prime component $Q$ of ${B''}^{=1}$ mapping onto $P$ (see~\cite[Theorem 2.3]{HP24}).
Let $G''$ be the boundary divisor obtained from $B''$
by increasing to one all the coefficients of the 
$f$-exceptional prime components of $B''$. 
Therefore $(X'',G''+M'')$ is a generalized dlt pair
and $K_{X''}+G''+M''\sim_\qq F\geq 0$ where $F$ is supported
on the union of $f$-exceptional prime divisors which are not
log canonical places of $(X'',B'',M'')$. Note that $F$ is vertical over $Z$. 
We run a $(K_{X''}+G''+M'')$-MMP over $Z$. 
This MMP terminates as $F$ is vertical over $Z$.
Indeed, up to $\rr$-linear equivalence over $Z$, the divisor
$F$ is of insufficient type over $Z$ so this MMP terminates by~\cite[Lemma 2.6]{Hua25} (see also~\cite[Lemma 2.9]{Lai11}). 
Let's call $g\colon X''\dasharrow Y$ the outcome of this MMP.
We denote by $B_Y$ (resp. $F_Y$, $M_Y$ and $G_Y$) 
the push-forward of $B''$ (resp. $F$, $M''$ and $G''$) to $Y$. 
Let $\pi_Y\colon Y \rightarrow Z$ be the induced projective morphism.
Thus, we have a commutative diagram
\[
\xymatrix{ 
X' \ar@{-->}[d]_-{\pi} & X'' \ar[d]_-{\pi'} \ar[l]_-{f}\ar@{-->}[r]^-{g} & Y \ar[ld]^-{\pi_Y} \\
Z & Z \ar[l]^-{{\rm Id}_Z} & 
}
\]
We conclude $K_Y+G_Y+M_Y\sim _{\qq, Z}0$ and hence that the effective divisor $F_Y$ is the pull-back of a divisor $F_{Z}$ from $Z$.
By construction, the effective divisor $F_{Z}$ does not contain any component of $B_{Z'}^{=1}$ in its support.
We apply the canonical bundle formula for $(Y,G_Y+M_Y)$ and obtain 
\[
\pi_Y^*(K_{Z}+G_{Z}+M_{Z})=K_{Y}+G_Y+M_Y,
\]
where $G_{Z}$ is an effective divisor. 
Note that the effective divisor $G_{Z}$ can be obtained from $B_{Z}$ by increasing to one
all the coefficients of its prime components with support in $\supp(F_{Z})$.
Furthermore, the divisor $K_Y+G_Y+M_Y$ is $\qq$-equivalent to an effective divisor
which is supported on $\supp(F_Y)$. 
We may run a $(K_{Z}+G_{Z}+M_{Z})$-MMP which terminates with a good minimal model 
$Z\rightarrow Z'$ (see~\cite[Section 2.5]{DHY23}). 
We run a $(K_Y+G_Y+M_Y)$-MMP over $Z'$. 
Note that $Y$ is a higher-dimensional compact complex K\"ahler variety, however,
the morphism $Y\rightarrow Z'$ is projective so this MMP can be performed
(see~\cite[Theorem 1.6]{Fuj22}). 
As $\Gamma_Y:=\pi_Y^*{\rm Ex}(Z\rightarrow Z')$ is exceptional over $Z'$, then this MMP must terminate
after contracting $\Gamma_Y$ (see~\cite[Lemma 2.6]{Hua25}).
We obtain a bimeromorphic contraction $Y\dashrightarrow Y'$
and a fibration $\pi_{Y'}\colon Y'\rightarrow Z'$. 
Henceforth, we have a commutative diagram 
\[
\xymatrix{ 
X' \ar@{-->}[d]_-{\pi} & X'' \ar[d]_-{\pi'} \ar[l]_-{f}\ar@{-->}[r]^-{g} & Y \ar[ld]^-{\pi_Y} \ar@{-->}[r]^-{h} & Y' \ar[d]^-{\pi_{Y'}} \\
Z & Z \ar[l] \ar[rr] & & Z'.
}
\]
As usual, we let $B_{Z'}$ (resp. $M_{Z'}, F_{Z'}$, and $G_{Z'}$) be the push-forward
of $B_{Z}$ (resp. $M_{Z},F_{Z}$, and $G_{Z}$) to the model $Z'$. 
Analogously, we let $B_{Y'}$ (resp. $M_{Y'},F_{Y'}$, and $G_{Y'}$)
be the push-forward of $B_Y$ (resp. $M_Y, F_Y$, and $G_Y$) 
and the following conditions are satisfied:
\begin{enumerate}
    \item the divisor $K_{Z'}+G_{Z'}+M_{Z'}$ is semiample, 
    \item the divisor $K_{Y'}+G_{Y'}+M_{Y'}$ is semiample, and
    \item the bimeromorphic map $Y'\dashrightarrow X'$ is a contraction.
\end{enumerate}
The divisor $K_{Y'}+G_{Y'}+M_{Y'}$, which is the pull-back
of $K_{Z'}+G_{Z'}+M_{Z'}$ via $\pi_{Y'}$, is $\qq$-linearly 
equivalent to an effective divisor which is exceptional over $X'$. 
Thus, we conclude that $K_{Y'}+G_{Y'}+M_{Y'}$ is $\qq$-linearly trivial
so $F_{Y'}=0$ and then $F_{Z'}=0$. 
This implies that $Y'\dashrightarrow X'$ only extracts log canonical
places of $(X',B')$. 
In particular, $(Y',B_{Y'}+M_{Y'})$ is a generalized log Calabi--Yau pair.
Furthermore, we conclude that $K_{Z'}+G_{Z'}+M_{Z'}\sim_\qq 0$.
As $F_{Z'}=0$ and $Z'$ is not uniruled, we get that $B_{Z'}=G_{Z'}\geq 0$.

As $Y'$ is the outcome of several MMP's, we conclude that $Y'$ is strongly $\qq$-factorial.
On the other hand, the variety $Y'$ has klt type singularities and so 
it has rational singularities. 
Let $\Sigma_{Y'}$ be the decomposition of $B_{Y'}$
obtained by taking the strict transform of $\Sigma'$ in $Y'$
and adding the reduced exceptional divisors of $Y'\dashrightarrow X'$.
Then, we have 
\[
c_{BC}(Y',B_{Y'};\Sigma_{Y'})=c_{\rm BC}(X',B';\Sigma')=3 
\text{ and }
\overline{c}(Y',B_{Y'};\Sigma_{Y'})=\overline{c}(X',B';\Sigma')=2.
\]
In particular, $\langle \Sigma_{Y'}\rangle$ is a subspace of $H^{1,1}_{\rm BC}(Y')$
of codimension one. 
Furthermore, we have a projective fibration $\pi_{Y'}\colon Y'\rightarrow Z'$.
Thus, we get a sequence 
\[
0 \rightarrow \pi_{Y'}^* H^{1,1}_{\rm BC}(Z') \rightarrow H^{1,1}_{\rm BC}(Y') 
\rightarrow H^{1,1}_{\rm BC}(F), 
\]
which is exact on the left and possibly not exact in the middle. 
Here, $F$ is the general fiber of $Y'\rightarrow Z'$
which turns out to be isomorphic to the general fiber of $X'\rightarrow Z$. 
As above, let $B_F$ be the restriction of $B_{Y'}$ to $F$
and $\Sigma_F$ be the restriction of $\Sigma_{Y'}$ to $F$. 
Note that
\[
\dim_\rr \langle \Sigma_F \rangle \leq  \dim_\rr  \langle \Sigma_{Y'}\rangle -
\dim_\rr  \left(\langle \Sigma_{Y'}\rangle \cap 
\pi_{Y'}^* H^{1,1}_{\rm BC}(Z')\right). 
\]
Thus, we conclude that 
\begin{equation}\label{eq3}
0\leq \overline{c}(F,B_F;\Sigma_F) 
 \leq \overline{c}(Y',B_{Y'};\Sigma_{Y'}) - \dim Z' - \dim_\rr  \left( \langle \Sigma_{Y'}\rangle \cap 
\pi_{Y'}^* H^{1,1}_{\rm BC}(Z')\right). 
\end{equation}
In particular, since $\overline{c}(Y',B_{Y'};\Sigma_{Y'}) = \dim Z'=2$, we have 
\[
\dim_\rr  \left( \langle \Sigma_{Y'}\rangle \cap 
\pi_{Y'}^* H^{1,1}_{\rm BC}(Z')\right)=0.
\]
It also follows that $h^{1,1}_{\rm BC}(Z')=1$ because if $h^{1,1}_{\rm BC}(Z')\geq 2$ then 
\[
\dim_\rr  \left( \langle \Sigma_{Y'}\rangle \cap 
\pi_{Y'}^* H^{1,1}_{\rm BC}(Z')\right) \geq 1
\]
which is impossible. 
If ${\rm Pic}(Z')\ne 0$, then
we conclude that $H^{1,1}_{\rm BC}(Y')={\rm Pic}(Y')_\rr$. 
This implies that $X$ is projective (see \cite[Lemma 2.42]{DH20}). 

Thus, we have that $h^{1,1}_{\rm BC}(Z')=1$ and $\rho(Z')=0$. 
In this case, we obtain $\overline{c}(F,B_F;\Sigma_F)=0$, so 
$F$ is a projective toric variety. 
As $Z'$ is not uniruled, we conclude that $B_{Z'}=M_{Z'}=0$
and that $Z'$ is canonical. Thus
$Z'$ is a canonical surface with $K_{Z'}\sim_\qq 0$.
Note that all the exceptional divisors of $Y'\dashrightarrow X'$
are log canonical places of $(X',B')$ which are vertical over $Z'$.
Since $Z'$ is a canonical surface, we conclude that 
$Y'\dashrightarrow X'$ is a small bimeromorphic modification. 
Thus, it suffices to take $W:=Z'$ to conclude the proof.
\end{proof}

\section{Examples}
In this section, we provide two examples of non-projective singular K3 surfaces of Picard rank zero and one-dimensional first Bott-Chern cohomology.

\begin{example}\label{ex:1}
{\em 
Let $f\colon X\rightarrow \pp^1$ be a projective K3 surface with an elliptic fibration. 
From~\cite[Theorem 1.1, Table 2, Example 1]{SZ01}, 
we may assume that $X$ has a section $C$ and six reducible fibers 
which are cycles of $4$ rational curves. 
For each such fiber, we may contract a chain of three rational curves
such that no chain intersects the section $C$.
We obtain a projective birational morphism $\phi\colon X\rightarrow X'$ 
where $X'$ has six $A_3$ singularities and a $(-2)$-curve $\phi(C)\subset X'$
which is disjoint from the singular points. 
We may blow-down $\phi(C)$ to obtain a projective singular K3 surface $X'\rightarrow X''$.
Note that $X''$ has Picard rank one and seven singular points $6A_3 + A_1$.
Let $\mathcal{L}\subset H^{1,1}(X)$ be the sublattice of rank 19 generated by the exceptional curves of $X\rightarrow X''$. 
Then, $\mathcal{L}$ is a negative-definite lattice of rank $19$.
We consider the Kuranishi deformation space $\mathcal{X}\rightarrow C$
consisting of deformations $X'$ of $X$ for which $\mathcal{L}\subset {\rm Pic}(X')$. 
As $\mathcal{L}$ has rank $19$, we conclude that $C$ is one-dimensional
and as $\mathcal{L}$ is negative-definite, we conclude that the general fiber
$\mathcal{X}_c$ of $\mathcal{X}\rightarrow C$ is non-projective. 
Since $\mathcal{L}\subseteq {\rm Pic}(\mathcal{X}_c)$, we deduce that we can contract
a configuration of $19$ $(-2)$-curves from $\mathcal{X}_c$ to obtain a projective bimeromorphic morphism $\mathcal{X}_c\rightarrow \mathcal{X}'_c$. 
We let $Z:=\mathcal{X}'_c$. 

The surface $Z$ is a singular K3 surface which is non-projective 
and satisfies $h^{1,1}_{\rm BC}(Z)=1$ and $\rho(Z)=0$. 
Therefore, we have $c_{\rm BC}(Z)=3$. 
}
\end{example}

\begin{example}\label{ex:2}
{\em 
Let $f\colon X \rightarrow \pp^1$ be a projective K3 surface with an elliptic fibration. 
From~\cite[Theorem 1.1, Table 2, Example 54]{SZ01}, we may assume that $X$ has a section $C$ and two reducible fibers which consist of cycles of $10$ rational curves. 
Furthermore, there is a section $C$ of $X$.
For each of such reducible fibers, we may contract $9$ of the rational curves and obtain
a projective birational contraction $X\rightarrow X'$ such that $X'$ has two $A_9$ singularities. Furthermore, the surface $X'$ has a $(-2)$-curve which is the image of the section of $X\rightarrow \pp^1$. 
We may contract the $(-2)$-curve $X'\rightarrow X''$ to obtain a singular projective
surface of Picard rank one with three singular points $2A_9+A_1$. 
Let $\mathcal{H}\subset H^{1,1}(X)$ be the sublattice of the Picard lattice generated
by the exceptional curves of $X\rightarrow X''$.
Let $\mathcal{X}\rightarrow C$ be the one-dimensional Kuranishi deformation space 
consisting of deformations $X'$ of $X$ for which $\mathcal{H}\subset {\rm Pic}(X')$. 
As $\mathcal{H}$ has rank 19 the family $\mathcal{X}\rightarrow C$ is one-dimensional
and as $\mathcal{H}$ is negative-definite, the general element $\mathcal{X}_c$
of this family is non-projective. 
Since $\mathcal{H}\subset {\rm Pic}(\mathcal{X}_c)$, 
we deduce that we can contract a configuration of $19$ 
$(-2)$-curves from $\mathcal{X}_c$ to obtain a singular surface 
$\mathcal{X}_c\rightarrow \mathcal{X}_c'$. We set $Z:=\mathcal{X}_c'$. 

The surface $Z$ is a singular K3 surface which is non-projective and satisfies 
$h^{1,1}_{\rm BC}(Z)=1$ and $\rho(Z)=0$. Therefore, we have $c_{\rm BC}(Z)=3$. 
} 
\end{example}

\bibliographystyle{habbvr}
\bibliography{bib}
\end{document}